\documentclass[11pt]{amsart}
\usepackage{amssymb}
\usepackage{verbatim}

\newtheorem{thm}{Theorem}
\newtheorem{lemma}[thm]{Lemma}

\theoremstyle{remark}
\newtheorem*{remark}{Remark}

\newcommand{\Z}{{\mathbb Z}}

\newcommand{\F}{{\mathbb F}}
\newcommand{\abs}[1]{\lvert#1\rvert}
\renewcommand{\bar}[1]{#1\llap{$\overline{\phantom{\rm#1}}$}}
\newcommand\as[1]{\renewcommand\arraystretch{#1}}

\begin{document}



\title{$p^k$-torsion of genus two curves over $\F_{p^m}$}

\author{Michael E. Zieve}
\address{
Michael E. Zieve\\
Department of Mathematics\\
Hill Center--Busch Campus\\
Rutgers, The State University of New Jersey\\
110 Frelinghuysen Road\\
Piscataway, NJ 08854--8019\\
USA
}
\email{zieve@math.rutgers.edu}
\urladdr{www.math.rutgers.edu/$\sim$zieve}

\date{August 29, 2008}

\subjclass{14H40}

\begin{abstract}
We determine the isogeny classes of abelian surfaces over $\F_q$ whose
group of $\F_q$-rational points has order divisible by $q^2$.
We also solve the same problem for Jacobians of genus-$2$ curves.
\end{abstract}

\maketitle


In a recent paper \cite{a}, Ravnsh{\o}j proved: if $C$ is a genus-$2$
curve over a prime field $\F_p$, and if one assumes that the endomorphism
ring of the Jacobian $J$ of $C$ is the ring of integers in a primitive
quartic CM-field, and that the Frobenius endomorphism
of $J$ has a certain special form, then $p^2\nmid\#J(\F_p)$.
Our purpose here is to deduce this conclusion under less
restrictive hypotheses.  We write $q=p^m$ where $p$ is prime,
and for any abelian variety $J$ over $\F_q$ we let $P_J$ denote
the \emph{Weil polynomial} of $J$, namely the characteristic
polynomial of the Frobenius endomorphism $\pi_J$ of $J$.
As shown by Tate \cite[Thm.~1]{tate}, two abelian varieties over $\F_q$ are
isogenous if and only if their Weil polynomials are identical.  Thus, the
following result describes the isogeny classes of abelian surfaces $J$ over
$\F_q$ for which $q^2\mid\#J(\F_q)$.

\begin{thm} \label{av}
The Weil polynomials of abelian surfaces $J$ over\/ $\F_q$ satisfying
$q^2\mid\#J(\F_q)$ are as follows\emph{:}
\begin{enumerate}
\item[(1.1)] $X^4 + X^3 - (q+2)X^2 + qX + q^2$ (if $q$ is odd and $q>8$);
\item[(1.2)] $X^4 - X^2 + q^2$;
\item[(1.3)] $X^4 - X^3 + qX^2 - qX + q^2$ (if $m$ is odd or $p\not\equiv 1\bmod{4}$);
\item[(1.4)] $X^4 - 2X^3 + (2q+1)X^2 - 2qX + q^2$;
\item[(1.5)] $X^4 + aX^3 + bX^2 + aqX + q^2$, where $(a,b)$ occurs in the same
row as $q$ in the following table:
\begin{center}
\as{1.1}
\begin{tabular}{|c|l|}
\hline
$q$ & $(a,b)$                       \\
\hline\hline
$13$ & $(9,42)$ \\ \hline
$9$ & $(6,20)$ \\ \hline
$7$ & $(4,16)$ \\ \hline
$5$ & $(3,6)$ \emph{or} $(8,26)$ \\ \hline
$4$ & $(2,5)$, $(4,11)$, \emph{or} $(6,17)$ \\ \hline
$3$ & $(1,4)$, $(3,5)$, \emph{or} $(4,10)$ \\ \hline
$2$ & $(0,3)$, $(1,0)$, $(1,4)$, $(2,5)$, \emph{or} $(3,6)$ \\
\hline
\end{tabular}
\end{center}
\end{enumerate}
\end{thm}

The special form required of the Frobenius endomorphism in \cite{a}
has an immediate consequence for the shape of its characteristic polynomial,
and by inspection the above polynomials do not have the required shape.
Thus the main result of \cite{a} follows from the above result.

Our proof of Theorem~\ref{av} relies on the classical results of
Tate (\cite[Thm.~1]{tate} and \cite[Thm.~8]{tate2}) and
Honda \cite{honda} describing the Weil polynomials of abelian varieties over
finite fields.  An explicit version of their results in the case of
simple abelian surfaces was given by R\"uck~\cite[Thm.~1.1]{rueck};
together with the analogous results of Waterhouse~\cite[Thm.~4.1]{waterhouse}
for elliptic curves, this yields the following:

\begin{lemma} \label{honda}
The Weil polynomials of abelian surfaces over $\F_q$ are
precisely the polynomials $X^4+aX^3+bX^2+aqX+q^2$, where $a,b\in\Z$ satisfy
$\abs{a}\le 4\sqrt{q}$ and $2\abs{a}\sqrt{q}-2q\le b \le \frac{a^2}{4}+2q$,
and where $a$, $b$, and the values $\Delta:=a^2-4(b-2q)$ and
$\delta:=(b+2q)^2-4qa^2$ satisfy one of the conditions \emph{(2.1)--(2.4)} below:
\begin{enumerate}
\item[(2.1)] $v_p(b)=0$;
\item[(2.2)] $v_p(b)\ge m/2$ and $v_p(a)=0$, and either $\delta=0$ or $\delta$ is a
non-square in the ring $\Z_p$ of $p$-adic integers;
\item[(2.3)] $v_p(b)\ge m$ and $v_p(a)\ge m/2$ and $\Delta$ is a square in $\Z$, and
if $q$ is a square and we write $a=\sqrt{q}a'$ and $b=qb'$ then
\begin{align*}
p\not\equiv 1\bmod{4}&\quad\text{if $b'=2$} \\
p\not\equiv 1\bmod{3}&\quad\text{if $a'\not\equiv b'\bmod{2}$;}
\end{align*}
\item[(2.4)] the conditions in one of the rows of the
following table are satisfied:
\begin{center}
\as{1.2}
\begin{tabular}{|c|l|}
\hline
$(a,b)$              & \emph{Conditions on $p$ and $q$}                       \\
\hline\hline
$(0,0)$              & \emph{$q$ is a square and $p\not\equiv1\bmod8$, or}    \\
                     & \emph{$q$ is a non-square and $p \ne 2$}               \\
\hline
$(0,-q)$             & \emph{$q$ is a square and $p\not\equiv1\bmod{12}$, or} \\
                     & \emph{$q$ is a non-square and $p \ne 3$}               \\
\hline
$(0,q)$              & \emph{$q$ is a non-square}                             \\
\hline
$(0,-2q)$            & \emph{$q$ is a non-square}                             \\
\hline
$(0,2q)$             & \emph{$q$ is a square and $p \equiv 1 \bmod4$}         \\
\hline
$(\pm \sqrt{q},q)$   & \emph{$q$ is a square and $ p \not \equiv 1 \bmod5$}   \\
\hline
$(\pm \sqrt{2q},q)$  & \emph{$q$ is a non-square and $p=2$}                   \\
\hline
$(\pm 2\sqrt{q},3q)$ & \emph{$q$ is a square and $p \equiv 1 \bmod3$}         \\
\hline
$(\pm \sqrt{5q},3q)$ & \emph{$q$ is a non-square and $p=5$}                   \\
\hline
\end{tabular}
\vskip 1em
\end{center}
\end{enumerate}
Moreover, the surface $J$ is simple if and only if either
\begin{itemize}
\item $\Delta$ is a non-square in $\Z$; or
\item $(a,b)=(0,2q)$ and $q$ is a square and $p\equiv 1\bmod{4}$; or
\item $(a,b)=(\pm 2\sqrt{q},3q)$ and $q$ is a square and $p\equiv 1\bmod{3}$.
\end{itemize}
The $p$-rank of $J$ (namely, the rank of the $p$-torsion subgroup of
$J(\bar{\F}_q))$ is $2$ in \emph{(2.1)}, $1$ in \emph{(2.2)}, and $0$ in
\emph{(2.3)} and \emph{(2.4)}.
\end{lemma}

\vspace{.1in}
\begin{proof}[Proof of Theorem~\ref{av}]
As shown by Weil \cite{weil}, for any abelian surface $J$ over $\F_q$, the Weil
polynomial $P_J$ is a monic quartic in
$\Z[X]$ whose complex roots have absolute value $\sqrt{q}$.
In particular, $\#J(\F_q)=\deg(\pi_J-1)=P_J(1)\le (\sqrt{q}+1)^4$, so if
$\#J(\F_q)=cq^2$ with $c\in\Z$ then $c\le (1+q^{-1/2})^4$.  It follows that
$c=1$ unless $q\le 27$.  In light of the above lemma, there are just finitely
many cases to consider with $c>1$; we treated these cases using the computer
program presented at the end of this paper, which gave rise to precisely the
solutions in (1.5).  Henceforth assume $c=1$.

The Weil polynomials of abelian surfaces over $\F_q$ are the polynomials
$P(X):=X^4 + a X^3 + b X^2 + a q X + q^2$ occurring in the above lemma.
We must determine which of these polynomials satisfy $P(1)=q^2$, or
equivalently, $b=-1-a(q+1)$.  The inequality $-1-a(q+1)=b\le a^2/4+2q$
says that $q^2\le (a/2+q+1)^2$, and since $a/2+q+1\ge -2\sqrt{q}+q+1>0$,
this is equivalent to $q\le a/2+q+1$, or in other words $-2\le a$.
The inequality $2\abs{a}\sqrt{q}-2q\le b=-1-a(q+1)$ always holds if $a\in\{0,-1,-2\}$,
and if $a\ge 1$ it is equivalent to $a(\sqrt{q}+1)^2\le 2q-1$; since
$2q-1<2q<2(\sqrt{q}+1)^2$, this implies $a=1$, in which case
$(\sqrt{q}+1)^2\le 2q-1$ is equivalent to $q\ge 8$.

Condition (2.1) holds if and only if $a\not\equiv -1\bmod p$, or equivalently
either $a\in\{0,-2\}$ or both $a=1$ and $p\ne 2$.  This accounts for (1.1), (1.2),
and (1.4).

Condition (2.3) cannot hold, since $p\mid a$ implies $b\equiv -1\bmod{p}$.

The condition $v_p(b)\ge m/2$ says that
$a\equiv -1\bmod{p^{\lceil m/2\rceil}}$, or equivalently $a=-1$.
In this case, $b=q$ and $\delta=9q^2-4q$, so $\delta\ne 0$.
If $q$ is odd then $\delta$ is a square in $\Z_p$ if and only if
$\delta$ is a square modulo $pq$, or equivalently, $m$ is even and $-4$ is
a square modulo $p$, which means that $p\equiv 1\bmod{4}$.
If $q$ is even then $\delta$ is not a square in $\Z_2$, since for
$q\le 8$ we have $\delta\in\{28, 128, 544\}$, and for $q>8$ we have
$\delta\equiv -4q\bmod{16q}$.  Thus (2.2) gives rise to (1.3).

Finally, if $a=-2$ then $b=2q+1$, and if $a=0$ then $b=-1$, so in either
case $q\nmid b$.  Thus (2.4) cannot hold, and the proof is complete.
\end{proof}


Next we determine which of the Weil polynomials in (1.1)--(1.5) occur for
Jacobians.  We use the classification of Weil polynomials of Jacobians of
genus-$2$ curves.  This classification was achieved by the
combined efforts of many mathematicians, culminating in the following result
\cite[Thm.~1.2]{b}:

\begin{lemma} \label{nonjac}
Let $P_J=X^4+aX^3+bX^2+aqX+q^2$ be the Weil polynomial of an abelian surface
$J$ over\/ $\F_q$.
\begin{enumerate}
\item  If $J$ is simple then $J$ is not isogenous to a Jacobian if
and only if the conditions in one of the rows of the following table are met:

\begin{center}
\as{1.2}
\begin{tabular}{|l|l|}
\hline
\emph{Condition on $p$ and $q$} &
 \emph{Conditions on $a$ and $b$}\\
\hline\hline
---                                 & \emph{$a^2-b = q$ and $b<0$ and} \\
                                     & \emph{all prime divisors of $b$ are $1\bmod3$}\\ \hline
 ---                                 & \emph{$a = 0$ and $b = 1-2q$} \\ \cline{1-2}
 $p>2$                               & \emph{$a = 0$ and $b = 2-2q$} \\ \hline
 \emph{$p\equiv 11\bmod 12$ and $q$ square} & \emph{$a = 0$ and $b = -q$}   \\ \cline{1-2}
 \emph{$p= 3$ and $q$ square}               & \emph{$a = 0$ and $b = -q$}   \\ \cline{1-2}
 \emph{$p= 2$ and $q$ non-square}            & \emph{$a = 0$ and $b = -q$}   \\ \cline{1-2}
 \emph{$q=2$ or $q=3$}                      & \emph{$a = 0$ and $b = -2q$}  \\ \hline
\end{tabular}
\end{center}
\vspace{1ex}

\vspace{.1cm}
\item If $J$ is not simple then there are integers $s,t$ such that 
$P_J=(X^2-sX+q)(X^2-tX+q)$, and $s$ and $t$ are unique if we require that
$\abs{s}\ge\abs{t}$ and that if $s=-t$ then $s\ge 0$.
For such $s$ and $t$, $J$ is not isogenous to a Jacobian if and only if the conditions in one of
the rows of the following table are met:
\end{enumerate}
\vspace{.2cm}

\begin{center}
\as{1.2}
\begin{tabular}{|c|l|l|}
\hline
\emph{$p$-rank of $J$} & \emph{Condition on $p$ and $q$} & \emph{Conditions on $s$ and $t$}\\
\hline\hline
 --- & ---                     & \emph{$|s-t| = 1$}                           \\ \hline
 $2$ & ---                     & \emph{$s=t$ and $t^2 - 4q\in \{-3,-4,-7\}$}  \\ \cline{2-3}
     & $q=2$                   & \emph{$s=1$ and $t=-1$}             \\ \hline
 $1$ & \emph{$q$ square}              & \emph{$s^2 = 4q$ and $s-t$ squarefree}       \\ \hline
     & \emph{$p>3$}                   & \emph{$s^2 \neq t^2$}                        \\ \cline{2-3}
     & \emph{$p=3$ and $q$ non-square} & \emph{$s^2 = t^2 = 3q$}                      \\ \cline{2-3}
 $0$ & \emph{$p=3$ and $q$ square}    & \emph{$s-t$ is not divisible by $3\sqrt{q}$} \\ \cline{2-3}
     & \emph{$p=2$}                   & \emph{$s^2 - t^2$ is not divisible by $2q$}  \\ \cline{2-3}
     & \emph{$q=2$ or $q=3$}          & \emph{$s = t$}                               \\ \cline{2-3}
     & \emph{$q=4$ or $q=9$}          & \emph{$s^2 = t^2 = 4q$}                      \\ \hline
\end{tabular}
\end{center}
\end{lemma}

\vspace{.3cm}

\begin{thm}
The polynomials in \emph{(1.1)--(1.5)} which are not Weil polynomials of Jacobians
are precisely the polynomials $X^4+aX^3+bX^2+aqX+q^2$
where $q$ and $(a,b)$ satisfy the conditions in one of the rows of the following table:
\begin{center}
\begin{tabular}{|c|l|}
\hline
$q$ & $(a,b)$                       \\
\hline\hline
$5$ & $(8,26)$ \\ \hline
$4$ & $(6,17)$ \\ \hline
$2$ & \emph{$(-2,5)$, $(0,3)$, $(1,4)$, $(2,5)$, or $(3,6)$} \\ \hline
\end{tabular}
\end{center}
\end{thm}
\vspace{.2cm}

\begin{proof}
Let $J$ be an abelian surface over $\F_q$ whose Weil polynomial
$P_J=X^4+aX^3+bX^2+aqX+q^2$ satisfies one of (1.1)--(1.5).
In each case, $a^2-b\ne q$, and if $a=0$ then $b\in\{-1,3\}$, so if $J$ is
simple then Lemma~\ref{nonjac} implies $J$ is isogenous to a Jacobian.

Henceforth assume $J$ is not simple, so
$P_J=(X^2-sX+q)(X^2-tX+q)$ where $s,t\in\Z$; we may assume that
$\abs{s}\ge\abs{t}$, and that $s\ge 0$ if $s=-t$.
Note that $a=-s-t$ and $b=2q+st$, so $(X-s)(X-t)=X^2+aX+b-2q$.  In particular,
$\Delta:=a^2-4(b-2q)$ is a square, say $\Delta=z^2$ with $z\ge 0$.

Suppose $P_J$ satisfies (1.1), so $\Delta=12q+9$.  Then $(z-3)(z+3)=12q$ is
even, so $z-3$ and $z+3$ are even and incongruent mod~$4$,
whence their product is divisible by $8$ so $q$ is even, contradiction.

Now suppose $P_J$ satisfies (1.2), so $\Delta=8q+4$.  Then $(z-2)(z+2)=8q$, so
at least one of $z-2$ and $z+2$ is divisible
by $4$; but these numbers differ by $4$, so they are both divisible by $4$, whence
their product is divisible by $16$ so $q$ is even.  Thus $8q$ is a power of $2$
which is the product of two positive integers that differ by $4$, so $q=4$.
In this case, $(q,a,b,s,t)=(4,0,-1,3,-3)$, which indeed satisfies (1.2).
Moreover, (2.1) holds, so Lemma~\ref{honda} implies $J$ has $p$-rank $2$.
Since $\abs{s-t}=6\notin\{0,1\}$ and $q\ne 2$, Lemma~\ref{nonjac} implies
$J$ is isogenous to a Jacobian.

Now suppose $P_J$ satisfies (1.3), so $\Delta = 4q+1$.  Then $(z-1)(z+1)=4q$,
so $z-1$ and $z+1$ are even and incongruent mod~$4$, whence their product is
divisible by $8$, so $q$ is even.  Thus $4q$ is a power of $2$ which is the
product of two positive integers that differ by $2$, so $q=2$.
In this case, $(q,a,b,s,t)=(2,-1,2,2,-1)$, which indeed satisfies (1.3).
Moreover, (2.2) holds, so Lemma~\ref{honda} implies $J$ has $p$-rank $1$.
Since $\abs{s-t}=3\ne 1$ and $q$ is a non-square, Lemma~\ref{nonjac} implies
$J$ is isogenous to a Jacobian.

Now suppose $P_J$ satisfies (1.4), so $\Delta=0$ and $a\notin\{0,\pm 2\sqrt{q}\}$,
and thus Lemma~\ref{nonjac} implies $J$ is non-simple.  Here
$(a,b,s,t)=(-2,2q+1,1,1)$, so Lemma~\ref{honda} implies $J$ has $p$-rank $2$.
Since $s=t=1$, Lemma~\ref{nonjac} implies $J$ is isogenous to
a Jacobian if and only if $1-4q\not\in\{-3,-4,-7\}$, or equivalently $q=2$.
This gives rise to the first entry in the last line of the table.

Finally, if $P_J$ satisfies (1.5) then the result follows from Lemma~\ref{nonjac}
and Lemma~\ref{honda} via a straightforward computation.
\end{proof}

\begin{remark}
The result announced in the abstract of \cite{a} is false,
since its hypotheses are satisfied by every two-dimensional
Jacobian over $\F_p$.  This is because the abstract of \cite{a}
does not mention the various hypotheses assumed in the
theorems of that paper.
\end{remark}

\vspace{.3cm}
We used the following Magma \cite{Magma} program in the proof of
Theorem~\ref{av}.

\begin{verbatim}
for q in [2..27] do if IsPrimePower(q) then
Q:=Floor(4*Sqrt(q)); M:=Floor((Sqrt(q)+1)^4/q^2);
for c in [2..M] do
for a in [-Q..Q] do b:=-1-a*(q+1)+(c-1)*q^2;
if b le (a^2/4)+2*q and 2*Abs(a)*Sqrt(q)-2*q le b then
p:=Factorization(q)[1,1]; m:=Factorization(q)[1,2];
Delta:=a^2-4*(b-2*q); delta:=(b+2*q)^2-4*q*a^2;
  if GCD(b,p) eq 1 then <q,a,b,c>;
  elif GCD(b,q) ge Sqrt(q) and GCD(a,p) eq 1 and
    (delta eq 0 or not IsSquare(pAdicRing(p)!delta)) then
    <q,a,b,c>;
  elif IsDivisibleBy(b,q) and GCD(a,q) ge Sqrt(q) and
    IsSquare(Delta) then
    if not IsSquare(q) then <q,a,b,c>;
    else sq:=p^((m div 2)); ap:=a div sq; bp:=b div q;
      if not ((bp eq 2 and IsDivisibleBy(p-1,4)) or
        (IsDivisibleBy(ap-bp,2) and IsDivisibleBy(p-1,3)))
        then <q,a,b,c>;
      end if;
    end if;
  elif (a eq 0 and b eq 0) then
    if ((IsSquare(q) and not IsDivisibleBy(p-1,8)) or
      (not IsSquare(q) and p ne 2)) then <q,a,b,c>;
    end if;
  elif (a eq 0 and b eq -q) then
    if ((IsSquare(q) and not IsDivisibleBy(p-1,12)) or
      (not IsSquare(q) and p ne 3)) then <q,a,b,c>;
    end if;
  elif a eq 0 and b in {q,-2*q} and not IsSquare(q) then
    <q,a,b,c>;
  elif a eq 0 and b eq 2*q and IsSquare(q) and
    IsDivisibleBy(p-1,4) then <q,a,b,c>;
  elif Abs(a) eq p^(m div 2) and b eq q and IsSquare(q) and
    not IsDivisibleBy(p-1,5) then <q,a,b,c>;
  elif Abs(a) eq p^((m+1) div 2) and b eq q and
    not IsSquare(q) and p eq 2 then <q,a,b,c>;
  elif Abs(a) eq 2*p^(m div 2) and b eq 3*q and IsSquare(q)
    and IsDivisibleBy(p-1,3) then <q,a,b,c>;
  elif Abs(a) eq p^((m+1) div 2) and b eq 3*q and
    not IsSquare(q) and p eq 5 then <q,a,b,c>;
  end if;
end if;
end for;
end for;
end if;
end for;
\end{verbatim}



\end{document}